\documentclass{amsart}
\usepackage{graphicx}
\usepackage{amssymb,amsmath,amsthm}
\usepackage{amsfonts}
\newtheorem{theorem}{Theorem}[section]
\newtheorem{corollary}[theorem]{Corollary}
\newtheorem{lemma}[theorem]{Lemma}
\newtheorem{proposition}[theorem]{Proposition}
\theoremstyle{definition}
\newtheorem{definition}[theorem]{Definition}
\theoremstyle{remark}

\numberwithin{equation}{section}

\newcommand{\cpmor}{\stackrel{cd}{-\hspace{-5pt}\twoheadrightarrow}}

\newcommand{\Cf}{{\bf C}}
\newcommand{\X}{\mathfrak X}
\newcommand{\D}{{[\ne]}}
\newcommand{\DE}{{\langle\ne\rangle}}
\newcommand{\I}{{\bf I}}
\newcommand{\romb}{\Diamond}
\begin{document}

\title{Topological Modal Logics with Difference Modality}%
\author{Kudinov Andrey}%
\address{}%
\email{kudinov--at--iitp--dot--ru}%



\begin{abstract}
We consider propositional modal logic with two modal operators
$\Box$ and $\D$. In topological semantics $\Box$ is interpreted
as an interior operator and $\D$ as difference. We show that
some important topological properties are expressible in this
language. In addition, we present a few logics and proofs of
f.m.p. and of completeness theorems.
\end{abstract}
\maketitle

\section{Introduction.} \label{Part0}
This paper deals with the topological semantics of modal logic.
The study of topological semantics of modal logic was started in
1944 by McKinsey and Tarski \cite{McKinsey}. Recently, this
topic has been attracting more attention partly due to
applications in AI (cf. \cite{RCC} and \cite{ManyDim}).
Reading the modal box as an interior operator one can easily show
that logic of all topological spaces is $\mathbf{S4}$. In
addition, McKinsey and Tarski proved that $\mathbf{S4}$ is also
the complete logic of the reals, Cantor space and indeed of any
metric separable space without isolated points (for a new proof of
this fact see \cite{Aiello}). Therefore, all these spaces are
modally equivalent, hence many natural properties of topological
spaces such as connectedness, density-it-itself and $T_1$ are
undefinable. For more information on spatial logics and spatial
reasoning see \cite{Aiello, Gabelaia, Shehtman,
ShehtmanDis}.

There are two ways of enriching the definability of a language: to
change semantics or to extend the language. According to the first
way, Esakia in \cite{Esakia} and Shehtman in
\cite{ShehtmanDis} considered the derivational logic (more
recent paper on this \cite{BEG}). According to the other way,
we can add the universal modality. In this new language we can
express connectedness (cf. \cite{Shehtman}).

In this paper, however, we add difference modality (or modality of
inequality) $\D$, interpreted as ``true everywhere except
here". Difference modality was suggested to use by several people
independently (in \cite{GPT} for one). More deeply this
modality and its interpretation in Kripke frames were studied in
\cite{Rijke}. It has been shown that difference modality
increase greatly the expressive power of a language (cf.
\cite{Goranko,Rijke}). The expressive power of this language
in topological spaces has been studied by Gabelia in
\cite{Gabelaia}, the author presented axioms that defines $T_1$
and $T_0$ spaces. Being added to the topological modal logic, the
difference modality allows us to express topological properties
that were unreachable before. The topological properties mentioned
in the end of the first paragraph became definable. The universal
modality is expressible as well in the following way: $[\forall] A
= \D A \land A$.

Here we also introduce three logics: $\mathbf{S4D}$,
$\mathbf{S4DS}$, and $\mathbf{S4DT_1S}$. We prove their f.m.p. and
following completeness theorems: $\mathbf{S4D}$ is complete with
respect to all topological spaces (Theorem \ref{theorS4D}),
$\mathbf{S4DS}$ is complete with respect to all dense-in-itself
topological spaces (Theorem \ref{TheorS4DS}), and
$\mathbf{S4DT_1S}$ is complete with respect to any
zero-dimensional dense-in-itself metric space (Theorem
\ref{theorS4DT1S}).

\section{Definitions and basic notions.}
Let us introduce some notations the reader will meet in this
paper. Assume that $B$ is a set, $R,R'\subseteq B\times B$ are
relations on $B$, then
$$
\begin{array}{l}
R\upharpoonright_A = R\cap (A\times A), \mbox{ for any } A\subseteq B;\\
Id_B = \{(x,x)\,|\,x\in B\};\\
R^+ = R\cup Id_B;\\
R\circ R' = \{(x,z)\,|\,\exists y\; (xRy\; \&\; yR'z)\};\\

R^1 = R,\; R^n = R^{n-1}\circ R;\\
R^* = \bigcup_{n=1}^\infty R^n.
\end{array}
$$
In this paper, we study propositional modal logics with two modal
operators, $\Box$ and $\D$. A formula is defined as follows:
$$
\phi ::= p\; | \;\bot \; | \; \phi \to \phi \; | \; \Box \phi \; |
\; \D \phi.
$$

The standard classic logic operators ($\vee, \wedge, \neg, \top,
\equiv$) are expressed in terms of $\to$ and $\bot$. The dual
modal operators $\romb, \DE$ are defined in the usual way as
$\romb A = \neg \Box \neg A$, $\DE A= \neg \D \neg A$
respectively. $[\forall] A$ stands for $\D A \land A$.

\begin{definition} A \emph{bimodal logic\/} (or \emph{a logic,\/} for short) is
a set of modal formulas closed under Substitution $\left
(\frac{A(p_i)}{A(B)}\right )$, Modus Ponens $\left (\frac{A,\,
A\to B}{B}\right )$ and two Generalization rules $\left
(\frac{A}{\Box A},\; \frac{A}{\D A}\right )$; containing all
classic tautologies and the following axioms
 $$
 \begin {array}{l}
 \Box (p\to q)\to (\Box p\to \Box q),\\
 \D(p\to q)\to (\D p\to \D q).
 \end{array}
$$

${\bf K_2}$ denotes \emph{the minimal bimodal logic}.
\end{definition}

Let L be a logic and let $\Gamma$ be a set of formulas, then $L +
\Gamma$ denotes the minimal logic containing $L$ and $\Gamma$. If
$\Gamma = \{ A \}$, then we write $L+A$ rather then $L+  \{ A \}$
.

In this paper, however, we consider a few additional axioms:
 $$
 \begin {array}{ccl}
 (B_D) &\quad& p \to \D\DE p\\
 (4^-_D) && (p \land \D p) \to \D\D p\\
 (T_\Box) && \Box p \to p      \\
 (4_\Box) && \Box p \to \Box \Box p\\
 (D_\Box) && [\forall] p \to \Box p\\
 (AT_1) && \D p \to \D \Box p\\
 (DS) && \D p \to \romb p\\
 \end {array}
 $$

The first two axioms are for $\D$ and they are from the paper
by de Rijke \cite{Rijke}. These axioms correspond to some basic
properties of inequality: symmetry and
pseudo-transitivity\footnote{In this paper relation $R$ is
pseudo-transitive iff $R^+$ is transitive. In some papers this
property calls weakly transitive (cf. \cite{Esakia, BEG})}
respectively.

The next two axioms are axioms for \textbf{S4}. These axioms have
well-known correspondence to the properties of topological
interior operator: $\I Y \subseteq Y$ and $\I Y \subseteq
\I\I Y$ respectively (where $Y$ is an arbitrary set). We
denote the interior and the closure operators by $\I$ and
$\Cf$ respectively.

Axiom $(D_\Box)$ is needed to connect $\Box$ and $\D$ and to
make sure that $[\forall]$ is the universal modality.

The meaning of the next two axioms will be explained later.

In this paper we study the following three logics:
$$
\begin {array}{l}
{\bf S4D}={\bf K_2}+\{B_D,4^-_D,D_\Box,T_\Box,4_\Box\},\\
{\bf S4DS}={\bf S4D}+DS,\\
{\bf S4DT_1S}={\bf S4DS}+AT_1.\\
\end{array}
$$

\section{Topological models.}

Let us define topological models.

\begin{definition} A \emph{topological model} is a pair $(\X,
\theta)$, where $\X$ is a topological space and $\theta$ is a
function assigning to each proposition letter $p$ a subset
$\theta(p)$ of $\X$. The function $\theta$ is called \emph{a
valuation}.
\end{definition}

\begin{definition}\label{Dmodality} The truth of a formula at a
point of a topological model is defined by induction:
$$
 \begin {array}{rlcl}
 (\mbox{i})&\X, \theta, x\models p & iff &  x\in \theta(p)\\
 (\mbox{ii})&\X, \theta, x\not\models \perp && \\
 (\mbox{iii})&\X, \theta, x\models \phi \to \psi & iff
 & \X, \theta, x\not\models \phi \;or\; \X, \theta, x\models\psi \\
 (\mbox{iv})&\X, \theta, x\models \Box\phi & iff &
 \mbox{there is a neighborhood $U$ of } x \mbox{ such }\\
 &&&\mbox{that for any }y\in U \quad\X, \theta, y\models \phi\\
 (\mbox{v})&\X, \theta, x\models \D \phi & iff
 & \X, \theta, y\models \phi \mbox{ for any } y\neq x\\

 \end{array}
 $$
\end{definition}

If $U$ is a subset of $\X$, then $\X, \theta, U \models A$
denotes that \mbox{$\X,\theta, x \models A$} for any $x\in U$.
A formula $A$ is called \emph{valid} in a topological space
$\X$ (notation: $\X\models A$), if it is true at any point
under any valuation. Also in notation $\X,\theta, x\models A$
we will omit the space and/or the valuation, if it is clear what
space and/or valuation we consider.

\begin{definition}
The \emph{D-logic} of a class of topological spaces $\mathcal{T}$
(in notation $L_D(\mathcal{T})$) is the set of all formulas that
are valid in all topological spaces from $\mathcal{T}$.
\end{definition}

Let us describe the classes of topological spaces axiomatized by
$(AT_1)$, $(DS)$.

\begin{definition}\label{defT1} A \emph{$T_1$-space} is a topological
space such that all its one-element subsets are closed.
\end{definition}

As we mentioned in introduction there is an axiom that defines
$T_1$ spaces in \cite{Gabelaia}, but it has a little bit
different form. And due to the next lemma they are equivalent on
topological spaces.

\begin{lemma}
Let $\X$ be a topological space then $\X \models AT_1$ iff
$\;\X$ is a $T_1$-space.
\end{lemma}

\begin{proof} ($\Rightarrow$) Ad absurdum. Suppose there exists $x\in \X$
such that $\{ x\}$ is not closed. Hence
 $\X - \{x\} \ne \I(\X-\{x\})$. Let $U=\X-\{x\}$. There exists

\begin{equation}
\label{yU} y\in U-\I U
\end{equation}

We take a valuation $\theta$ in $\X$ such that $\theta(p)=U$.
Then $x\models \D p$ but $x\models AT_1$; hence $x\models
\D\Box p$. Since $y\ne x$, we have $y\models \Box p$, which
means that $y$ together with some its neighborhood is in $U$. This
contradicts to (\ref{yU}).

($\Leftarrow$) Assume that $\X$ is a $T_1$-space. Let $\X,
\theta, x \models \D p$ then $\theta(p)\supseteq \X-\{x\}$.
We need to prove that $x \models \D \Box p$. It means that for
all $y\in \X-\{x\}$ $ y \models \Box p$. Take any $y\in
\X-\{x\}$. Since $\X-\{x\}$ is open, there exists an open
$U\ni y$ and $U\subseteq \X-\{x\}$. So $ U\models p$, then
$y\models \Box p $, hence $x \models \D \Box p$.

\end{proof}

\begin{definition}\label{DefDS} Let $\X$ be a topological space. A point $x\in
\X$ is called isolated, if $\{x\}$ is open. $\X$ is called
\emph{dense-in-itself}, if it has no isolated points.
\end{definition}

\begin{lemma}
Let $\X$ be a topological space then $\X\models DS$ iff
$\;\X$ is dense-in-itself.
\end{lemma}
\begin{proof} ($\Rightarrow$) Ad absurdum. Assume that $\X$ is not
dense-in-itself and $x\in \X$ is isolated.

Let us take a valuation $\theta$ in $\X$ such that
$\theta(p)=\X-\{x\}$; then $x \models \D p$. Since $\{x\}$
is open and $x\models \neg p$, it follows that $x \models \Box\neg
p$ or equivalently, $x \models \neg \romb p$. This contradicts
to the axiom $(DS)$.

($\Leftarrow$) Assume that $\X$ is dense-in-itself and
$(\X,\theta),x\models \D p$; then there are two cases: \\
(i) $\theta(p)=\X$, in this case it is obvious that
$(\X,\theta),x\models \romb p$; \\
(ii) $\theta(p) = \X-\{x\}$, then $(\X,\theta),x\models
\romb p$ since $x\in \Cf(\X-\{x\})=\X$.
\end{proof}

\section{Kripke frames and models.}

Kripke frames and models are well-known basic notions of modal
logic (cf. \cite{BRV} and \cite{Chagrov}).

\begin{definition} A \emph{Kripke frame}\/ is a tuple $F
= (W,R_1, \ldots R_n)$ such that
$$
\begin{array}{rl}
\mbox{(i)}& \mbox{$W$ is a non-empty set,}\\
\mbox{(ii)}&\mbox{$R_i$ for $i=1\ldots n$ are binary relations on
$W$.}
\end{array}
$$
\end{definition}

In this paper however, we consider Kripke frames with one or two
relations only. The first is denoted as $R$ and the second (if it
is present) --- as $R_D$.

\begin{definition}
A \emph{Kripke model}\/ is a pair $\mathcal{M} = (F, \theta)$,
where $F$ is a frame and $\theta$ is a valuation (a function from
the set of all proposition letters to the set of all subsets of
$W$).
\end{definition}

$\mathcal{M}, x \models A$ denotes that formula $A$ is true in
model $\mathcal{M}$ at point $x$; $\mathcal{M} \models A$ denotes
that $A$ is true at all points of model $\mathcal{M}$; $F \models
A$ denotes that $(F,\theta),x \models A$ for all valuations
$\theta$ and all points $x \in W$; $F, x \models A$ denotes that
$(F,\theta),x \models A$ for all valuations $\theta$. For a subset
$U\subseteq W$ $\mathcal{M}, U \models A$ denotes that for any
$x\in U$ $(\mathcal{M}, x \models A)$.

\begin{definition}
The \emph{logic} of a class of frames $\mathcal{F}$ (in notation
$L(\mathcal{F})$) is the set of all formulas that are valid in all
frames from $\mathcal{T}$. For a single frame $F$, $L(F)$ stands
for $L(\{F\})$.
\end{definition}

\begin{definition}  A frame $F$ is called \emph{a $\Lambda$-frame}
for a modal logic $\Lambda$, if $\Lambda \subseteq L(F)$.
\end{definition}

\begin{definition}
A \emph{$p$-morphism} from a Kripke frame $F=(W,R,R_D)$ onto a
Kripke frame $F'=(W',R',R_D')$ is a map $f:W\to W'$ satisfying the
following conditions:
\begin{enumerate}
  \item $f$ is surjective;
  \item $\forall x \forall y (xRy \Rightarrow f(x) R' f(y)$ and
  the same for $R_D$ and $R_D'$;
  \item $\forall x \forall z (f(x)R'z \Rightarrow \exists y (x R y \&
  f(y)=z))$ and the same for $R_D$ and $R_D'$.
\end{enumerate}
In notation: $f: F \twoheadrightarrow F'$.
\end{definition}

\begin{definition}
By cone $F^x$ we will understand the frame $$(W^x,
R\upharpoonright_{W^x}, R_D\upharpoonright_{W^x} ), $$ where $W^x
= ( R \cup R_D)^+(x)$. If for some $x$ $F=F^x$ then $F$ called
\emph{rooted}.
\end{definition}

The following two lemmas are well-known (cf. \cite{BRV} and
\cite{Chagrov}).

\begin{lemma}\label{LemGen} Let $F=(W, \ldots )$ be a Kripke frame, then
$$L(F) = \bigcap \{L(F^x\;|\;x\in W\}.$$
\end{lemma}

\begin{lemma}($p$-morphism Lemma)
$f: F \twoheadrightarrow F'$ implies $L(F)\subseteq L(F')$.
\end{lemma}

In this paper we consider only {\bf S4D}-frames. The axioms $B_D,
4^-_D,D_\Box,T_\Box,4_\Box$ put constraints on relations $R$ and
$R_D$. So from now on we assume that all Kripke frames satisfy the
following conditions:
\begin{itemize}
\item $R$ is reflexive (axiom $T_\Box$) and transitive ($4_\Box$),
\item $R_D$ is symmetric ($B_D$)
\item $R_D$ is pseudo-transitive ($4^-_D$),
\item $R\subseteq R_D \cup Id_W$
($D_\Box$).
\end{itemize}

Note that we can further assume that $R_D \cup Id = W\times W$,
because according to Lemma \ref{LemGen} we can consider only
generated subframes.

Now let us see what formulas $AT_1$ and $DS$ mean in a Kripke
frame.

Let $F=(W,R,R_D)$ be a {\bf S4D}-frame, then $Top(F)=Top(W,R)$
denotes the topological space on the set $W$ with the topology
$\left\{ R(V) \;|\; V\subseteq W \right\}$. For formulas with the
difference modality the validity in $F$ and $Top(F)$ may not be
equivalent.This is because $R_D$ could be not the real inequality
relation.

\begin{definition} Let $R$ be a transitive reflexive relation on
$W$. Then $x\in W$ is called \emph{$R$-minimal} (respectively
$R$-maximal), if for any $y$, $yRx$ (respectively $xRy$) implies
$x=y$.
\end{definition}

\begin{definition} Let $F=(W,R,R_D)$ be an
$\mathbf{S4D}$-frame; we say that $F$ is a \emph{$T_1$-frame} (or
has the \emph{$T_1$--property}), if all $R_D$-irreflexive points
are $R$-minimal.
\end{definition}

\begin{lemma}
\label{PropF} Let $F=(W,R,R_D)$ be $\mathbf{S4D}$-frame. Then
$F\models AT_1$ iff
$F$ is a $T_1$-frame. 
\end{lemma}
\begin{proof} ($\Rightarrow$) Suppose $F\models AT_1$ and
there exists an $R$-non-minimal and $R_D$-irreflexive point in
$F$. To be more specific, let $x$ and $y$ be two different points
such that $\neg x R_D x$ and $yRx$. Take a valuation $\theta$ such
that $\theta(p)=W-\{x\}$. Then $x\models \D p$ and $x\models
\neg p$, thus $y\models \romb \neg p$. Since $x\ne y$ and
$yRx$, we have $xR_Dy$, $x\models \DE\romb\neg p$. Hence
$x\models \neg\D\Box p$. This contradicts $x\models AT_1$.

($\Leftarrow$) Assume that $F$ is a $T_1$-frame and for some
valuation for $F$ we have $x\models\D p$. Let us show that
$x\models \D\Box p$. As we mentioned above, generated subframes
preserve validity, so we can assume that $F=F^x$ hence $R_D(x)\cup
\{x\} = W$. There are two possibilities:

1) $xR_Dx$. Then $y\models p$ for any $y\in W$, hence for all
$y\in W$ we have $y\models \Box p$; so $x\models \D\Box p$.

2) $\neg xR_Dx$. Then $y\models p$ for every $y\ne x$. By
assumption, $y\ne x$, $yRz$ implies $z\ne x$, hence $z\models p$.
So for any $y\ne x$ $y\models \Box p$; hence $x\models \D\Box
p$.
\end{proof}

\begin{definition} Let
$F=(W,R,R_D)$ be an $\mathbf{S4D}$-frame; we say that $F$ is a
\emph{$DS$-frame}, if every $R_D$-irreflexive point has an
R-successor (called just a \emph{successor} further on).
\end{definition}

\begin{lemma} Let $F=(W,R,R_D)$ be an $\mathbf{S4D}$-frame. Then
$F\models DS$ iff $F$ is a $DS$-frame.
\end{lemma}
\begin{proof} ($\Rightarrow$) Suppose $F \models DS$ and
there exists an $R_D$-irreflexive point $x$ without successors. We
take a valuation $\theta$ such that $\theta(p)=W-\{x\}$; then
$x\models \D p$ but $x\not\models \romb p$. This contradicts
$F \models DS$.

($\Leftarrow$) Suppose that every $R_D$-irreflexive point in $F$
has a successor. Let us prove that for any $x\in W$ $x\models DS$.
Suppose $(F,\theta),x\models \D p$, then there are two cases:
(i) $x$ is $R_D$-reflexive; then $\theta(p)=W$, and so $x \models
\romb p$ since $R$ is reflexive; (ii) $x$ is $R_D$-irreflexive,
then $\theta(p)\supseteq W-\{x\}$, and by our assumption, there
exists $y\ne x$ such that $xRy$; hence $y\models p$ and $x\models
\romb p$. \end{proof}

\section{Kripke completeness and finite model property.}

All our axioms are Sahlqvist formulas. So we easily obtain Kripke
completeness for logics \textbf{S4D}, $\bf S4DS$, $\bf S4DT_1S$.

Following the common way of proving f.m.p.\ we use filtration (cf.
\cite{Chagrov} and \cite{BRV}).

\begin{definition} Let  $M=(F,\theta)$ be a Kripke
model, where $F=(W, R, R_{D})$ is a Kripke frame and $\Psi$ is a
set of formulas closed under subformulas. Let $\approx_\Psi$ be
the equivalence relation on the elements of $W$ defined as
follows:
$$
w \approx_\Psi v \mbox{ iff for all $\phi$ in $\Psi$:} (M,w\models
\psi \mbox{ iff } M,v \models \phi).
$$
By $[w]$ we denote the equivalence class of $w$. Suppose
$M'=(F',\theta ')$ and $F'=(W',R',R_D')$ such that

\begin{enumerate}
\item $W'=W_\Psi=\{[w]\;|\; w\in W\}$.
\item \label{defil1} If $wRv$ then $[w]R'[v]$ (and
similarly for $R_{D}$),
\item \label{defil3}
If $[w]R'[v]$ then for all $\Box \phi\in \Psi$; $M,w\models\Box
\phi$ only if $M,v\models \phi$ (and similarly for $R_{D}$ and
$\D$).
\item \label{defil2}
$\theta '(p)=\{[w]\;|\; M,w\models p\}$, for all atomic symbols
$p$ in $\Psi$.
 \end{enumerate}
Then $M'$ is called a filtration of $M$ through $\Psi$.
\end{definition}

\begin{lemma} \emph{(Filtration Lemma)} Let $M'$ be a
filtration of $M$ through $\Psi$, then for any $x\in M_1$ and for
any $\psi\in \Psi$
$$
M,w\models \psi \iff M',[w] \models \psi.
$$
\end{lemma}

\begin{lemma}
\label{filtr} Let $F_1$ be an {\bf S4D}-frame, $M_1=(F_1,
\theta_1)$ a model, $\Psi$ a finite set of formulas closed under
subformulas. Then there exists a filtration $M_2$ of $M_1$ through
$\Psi$, such that $M_2 = (F_2,\theta_2)$ and $F_2$ is an {\bf
S4D}-frame.
\end{lemma}

\begin{proof} Let $M'=(W_\Psi, R', R_D', \theta ')$ be the minimal
filtration of $M$. The minimal filtration is well-known (cf.
\cite{Chagrov} or \cite{McKinsey}). Briefly, $[x]R'[y]$ iff
there exist $x'\in [x]$ and $y'\in [y]$ such that $x'R_1 y'$, and
the same for $R'_D$ and $R_{D1}$.

Let $R_2$ be the transitive closure of $R'$:
$$
R_2 = R'^*= \bigcup_{n\ge 1} R'^n;
$$
and let $R_{D2}$ be the pseudo-transitive closure of $R_{D}'$:
$$
R_{D2} = R_{D}'^* - (Id - R_{D}').
$$
Note that the only difference between the pseudo-transitive and
the transitive closure is that the irreflexive points remain
irreflexive.

One can easily see that the reflexivity of $R'$ is inherited by
$R_2$, and the reflexivity of $R'$ follows from the reflexivity of
$R_1$. In the same way the symmetry of $R_{D1}$ implies the
symmetry of $R_{D2}$. The transitivity of $R_2$ and the
pseudo-transitivity of $R_{D2}$ are provided by construction.
Next, we can easily show that $R'\subseteq R_{D}'\cup Id$; hence
$R_2\subseteq R_{D2}\cup Id$ holds.

To complete the proof, we have to show that the relations $R_2$
and $R_{D2}$ satisfy the definition of filtration. Since
Filtration Lemma for the minimal filtration and its transitive
closure in transitive logics are well-known (cf. \cite{BRV}),
we will only check $R_{D2}$.
\begin{enumerate}
\item For arbitrary $w, v \in W_1$ assume $v R_{D1} w$, let us
prove that $[v] R_{D2}[w]$. If $[v]\ne [w]$, the proof is the same
as for the transitive closure. So assume $[w]=[v]$; then
$[v]R'_{D}[w]$, and so $[v]R_{D2} [w]$.
\item Assume $[v] R_{D2} [w]$, let us prove that for all $\D \psi \in \Psi$;
$M_1,v\models\D \psi$ only if $M_1,w\models \psi$. If $[v]\ne
[w]$, then the proof is the same as for the transitive closure. If
$[w]=[v]$, then from $[v]R_{D2}[v]$ follows $[v]R'_{D}[v]$. But
$R'_D$ was already filtration.
\end{enumerate}

So we obtain a filtration that reduces $M_1$ to a finite model
over an $\mathbf{S4D}$-frame. \end{proof}

\begin{theorem}
\label{FinitApproc} Let $L$ be one of the logics:
$\mathbf{S4D}$, $\mathbf{S4DS}$, $\mathbf{S4DT_1S}$. Then $L$ has
the finite model property.
\end{theorem}
\begin{proof} Assume that $A$ is a formula such that
$A$ is not in $L$. Hence, $A$ is refuted in some generated
submodel $M_1=(W_1,R_1,R_{D1},\theta)$ of the canonical model of
logic $L$. Note that since $M_1$ is a generated submodel,
$R_{D1}\cup Id_{W_1}$ is the universal relation.

Let $\Psi$ be the set of all subformulas of formula $A$. By Lemma
\ref{filtr} there exists model $M_2=(F_2,\theta_2)$ such that
$F_2$ is a $\mathbf{S4D}$-frame and $M_2$ is a filtration of $M_1$
through $\Psi$.

Since $M_2$ is a filtration, $A$ is refuted in $M_2$. So it
remains to prove (if needed) the $T_1$--property and the
$DS$-property for $F_2$.

Let us prove that axiom $AT_1$ is valid in frame $F_2$. By Lemma
\ref{PropF}, it is sufficient to prove that for any $\eta$ such
that $ \neg \eta R_{D2}\eta$ there does not exist $\psi$ such that
$\psi\ne \eta$ and $\psi R_2\eta$.

Assume the contrary, i.e. there exists a point $\psi\ne \eta$ such
that $\psi R_2\eta$. Then consider their inverse images: $ [x] =
\eta ~\&~ [y] = \psi$. By construction of $R_2$ we obtain
$$
y\approx_\Psi y_0 R_1 z_0 \approx_\Psi y_1 \ldots y_k R_1 z_k
\approx_\Psi x
$$

Since $y\not\approx_\Psi x$, we can take maximal $l$ such that
$y_l\not\approx_\Psi z_l$. By transitivity of $\approx_\Psi$ we
conclude that $z_l \approx_\Psi x$ and $[z_l]=[x]=\eta$. Assume
that $z_l\ne x$, since $R_{D1}\cup Id$ is the universal relation
$z_l R_{D1} x$, hence $[z_l] R_{D2} [x]$; which contradicts $\lnot
\eta R_{D2} \eta$.

So $y_l \ne x~ \&~ y_l R_1 x(=z_l)$, at the same time reflexivity
is preserved under filtration; hence $\neg x R_{D1} x$. So we came
to a contradiction, because the generated subframe
$F_1=(W_1,R_1,R_{D1})$ of the canonical model has the
$T_1$--property.

Now let $[x]$ be an $R_{D2}$-irreflexive point. Hence, $x$ is also
an $R_{D1}$-irreflexive point, so for some $y$, $x R_1 y$ and $x
R_{D1} y$, then $[x]R_2 [y]$ and $[x]R_{D2} [y]$, so $[y]\ne [x]$;
hence $[x]$ is not maximal. \end{proof}

\section{Topological completeness.}
Let us define analogue of $p$-morphism for maps from topological
space onto finite $\mathbf{S4D}$-frame.

\begin{definition} Let $\X$ be a topological space and let
$F=(W,R,R_D)$ be a finite Kripke frame. A function $f:\X\to F$
is called a \emph{cd-$p$-morphism}, if it is surjective and
satisfies the following two conditions
\begin{equation}
\label{RCf} \Cf f^{-1}(w)=f^{-1}(R^{-1}(w)),
\end{equation}
\begin{equation}
\label{RDf} R_D^{-1}(f^{-1}(w))=f^{-1}(R_D^{-1}(w)),
\end{equation}
where $R_D= ``\ne"$ in $\X$ (in particular $R_D^{-1}(\{x\}) =
\X-\{x\}$). In notation $f:\X\cpmor F$.
\end{definition}

Note that since $f$ is surjective, (\ref{RDf}) is equivalent to
the following: if $w$ is $R_D$-irreflexive then $f^{-1}(w)$ is
one-element.

\begin{lemma}
\label{cdpmorph} If $F$ is a finite Kripke frame, $\X$ is a
topological space and $f: \X \cpmor F$ then
$L_D(\X)\subseteq L(F)$.
\end{lemma}

\begin{proof} Note that $f$ is $cd-p-$morphism and $\Cf$ distributes
over finite\footnote{$\Cf$ is not distributes over infinite
unions so finiteness of $F$ is essential.} unions. So for
$U\subseteq W$ we have

\begin{equation}
\label{Ufinite}
\begin{array}{l}
f^{-1}(R^{-1}(U)) = f^{-1} (\bigcup\limits_{w\in U} R^{-1}(w)) =
\bigcup\limits_{w\in U} f^{-1}(R^{-1}(w)) \stackrel{(\ref{RCf})}=\\
\qquad\qquad\qquad = \bigcup\limits_{w\in U} \Cf f^{-1}(w) =
\Cf f^{-1}(U)
\end{array}
\end{equation}

In other terms, $f$ is an interior map between topological spaces
$\X$ and $Top(F)$.

Similarly

\begin{equation}
\label{UfiniteD}
\begin{array}{l}
f^{-1}(R_D^{-1}(U)) = f^{-1} (\bigcup\limits_{w\in U} R_D^{-1}(w))
=
\bigcup\limits_{w\in U} f^{-1}(R_D^{-1}(w)) \stackrel{(\ref{RDf})}=\\
\qquad\qquad\qquad = \bigcup\limits_{w\in U} R_D^{-1} f^{-1}(w) =
R_D^{-1} f^{-1}(U).
\end{array}
\end{equation}

Now let $\theta$ be an arbitrary valuation on the frame $F$. Take
a valuation $\Theta$ on $\X$ such that $\Theta(p) =
f^{-1}(\theta(p))$. Then a standard inductive argument shows that
for any formula $\phi$
\begin{equation}\label{eqvaluations}
\Theta(\phi)=f^{-1}(\theta(\phi)),
\end{equation}
where $\theta(\phi)= \left\{v\left| (F,\theta),v\models \phi
\right.\right\}$ and $\Theta(\phi)= \left\{x\left|
(\X,\Theta),x\models \phi \right.\right\}$.

For this proof we rewrite all formulas using $\romb$ and
$\DE$ (rather then $\Box$ or $\D$).

There are only two nontrivial cases:

i) $\phi \equiv \romb \psi$. Then
$$
f^{-1}(\theta(\romb\psi)) =
f^{-1}(R^{-1}(\theta(\psi)))\stackrel{(\ref{Ufinite})}= \Cf
f^{-1}(\theta(\psi)) \stackrel{induction}= \Cf\Theta(\psi) =
\Theta(\romb \psi).
$$

ii) $\phi\equiv \DE \psi$. Then
$$
f^{-1}(\theta(\DE\psi)) = f^{-1}(R_{D}^{-1}(\theta(\psi)))
\stackrel{(\ref{UfiniteD})}= R_{D}^{-1} f^{-1}(\theta(\psi))
\stackrel{induction}= R_{D}^{-1}\Theta(\psi) = \Theta(\DE
\psi).
$$

Now if $\phi \not\in \mathbf{L}(F)$, there exists a valuation
$\theta$ such that $\theta (\phi) \ne W$. By
(\ref{eqvaluations}) $\Theta(\phi)=f^{-1}(\theta(\phi))$, and
so $\Theta (\phi)\ne \X$ since $f$ is subjective. Thus $\phi
\not\in \mathbf{L}(\X)$.
\end{proof}

The following proposition uses ideas from
\cite{Rijke,Esakia}

\begin{proposition}
Let $F=(W,R,R_D)$ be a $\mathbf{S4D}$-Kripke frame, $R_D \cup Id_W
= W\times W$ . There exists $\mathbf{S4D}$-Kripke frame
$F'=(W',R',R_D')$, such that $F'\twoheadrightarrow F$ and $x'R_D'
y'$ iff $x'\ne y'$.
\end{proposition}
\begin{proof}
Let us put $W^0 = \{x\in W| x R_D x\}$ and $W^\times = W-W^0$.
Then
\begin{equation} \label{WWW}
W' = W^\times \cup W^0 \times \{0,1\}
\end{equation}
Let us define the function $f:F'\to F$ such that
$$
f(x')=\left\{
\begin{array}{ll}
x, &\mbox{if } x'=x\in W^\times;\\
x, &\mbox{if } x'=(x,i);
\end{array}\right.
$$
and the relation $R'$:
$$
x'R'y' \iff f(x')R f(y')
$$

Let us prove that $f$ is a $p$-morphism.
\begin{enumerate}
  \item Obviously $f$ is surjective.
  \item Assume that $x' R' y'$ then by definition of $R'$ $f(x')
  R f(y')$. Assume that $x' R_D' y'$ (or $x'\ne y'$), $f(x')=x$
  and $f(y')=y$. If $x\ne y$ then $x R_D y$. If $x=y$ then
  $x'=(x,0)$ and $y'=(x,1)$ (or vice verse); using (\ref{WWW})
  we conclude that $xR_D x=y$.

  \item Assume that $f(x') R y$. If $y \in W^0$ then
  $y'=(y,0)$ or $y'=y$ otherwise. Easy to see that $f(y')=y$ and
  $x'R'y'$. Assume that $f(x') R_D y$. Case when $f(x') \ne y$ is
  obvious so let $f(x') = y$. It means that $y\in W^0$ and
  $x'=(y,i)$. So we put $y'=(y, (i+1) \;\mathrm{mod}\; 2)$ and this will do.
\end{enumerate}
\end{proof}

\begin{corollary}\label{cor1}
Let $\mathcal{C}$ be the class $\mathbf{S4D}$-frames of the form
$F=(W,R,\ne)$ then $\mathbf{S4D}$ is complete with respect to
$\mathcal C$.

\end{corollary}

It is easy to show that for any $\mathbf{S4D}$-frame $F=(W,R,\ne)$
$$
Top(F)\cpmor F
$$
but we can prove a stronger statement:

\begin{lemma}\label{LemFTop}
Let $(F,\theta)$ be a Kripke model then for any formula $A$ and
$x\in W$
$$
F,\theta,x\models A \iff Top(F),\theta,x\models A
$$
\end{lemma}
\begin{proof}
By induction on the complexity of $A$. The only case that is not
trivial or classical is when $A=\D B$.
$$
F,\theta,x\models \D B \iff \forall y \left( y\ne x \Rightarrow
F,\theta,y\models B\right)
$$
but by induction it holds iff
$$
\forall y \left( y\ne x \Rightarrow Top(F),\theta,y\models
B\right) \iff Top(F),\theta,x\models \D B
$$
\end{proof}

\begin{theorem}\label{theorS4D}
$\mathbf{S4D}$ is the D-logic of all topological spaces.
\end{theorem}
\begin{proof}
Let $A$ be a formula that is not in $\mathbf{S4D}$. Then by
Corollary \ref{cor1} there exists a Kripke frame $F=(W,R,\ne)$
such that $F\not\models A$. By Lemma \ref{LemFTop} we obtain
$Top(F)\not\models A$
\end{proof}

\begin{proposition} \label{prop2}
Let $F=(W,R,\ne)$ be a DS-frame, then $Top(F)$ is a
dense-in-itself topological space.
\end{proposition}
\begin{proof}
In $Top(F)$ the least open neighborhood of point $x$ is $R(x)$.
Since $F$ is a DS-frame, $R(x)-\{x\}\ne\varnothing$; hence
$Top(F)$ is dense-in-itself.
\end{proof}

\begin{theorem}\label{TheorS4DS}
$\mathbf{S4DS}$ is logic of all dense-in-itself topological
spaces.
\end{theorem}
\begin{proof}
From Theorem \ref{FinitApproc} we know that $\mathbf{S4DS}$ is
complete with respect to all finite DS-frames. Now we can apply
Proposition \ref{prop2} and Lemma \ref{LemFTop}.
\end{proof}

If a logic contains the axiom ($AT_1$) then we cannot use the
above methods. Indeed if $F=(W,R,\ne)$ and $F\models AT_1$, then
$R=Id_W$. The logic of such frames will be the logic of isolated
points. So we need to find more sophisticated ways.

Recall a few definitions.

\begin{definition} A non-empty topological space $\X$ is called
\emph{zero-dimensional} if clopen sets constitute its open base.
\end{definition}

\begin{definition}
A pair $(X,\rho)$ called metric space if $X$ is a set and $\rho$
is a function from $X\times X$ onto $\mathbb{R}$, such that
$\rho(x,y)\ge 0$, $\rho(x,y)=0$ iff $x=y$, $\rho(x,y)=\rho(y,x)$,
and $\rho(x,y)+\rho(y,z)\ge \rho(x,z)$.
\end{definition}

On metric space can be defined natural topology based on open
balls: $\{y\;|\; \rho(x,y)<r\}$.

\begin{theorem}\label{theorS4DT1S}
$\mathbf{S4DT_1S}$ is compete with respect to any zero-dimensional
dense-in-itself metric space.
\end{theorem}

\begin{proof}
Let $\X$ be a zero-dimensional dense-in-itself metric space and
$\rho$ is the distance in it; $O(x,r)$
 denotes the open ball $\{y\in\X | \rho(x,y)<r)$.

We know from Theorem \ref{FinitApproc} that $\mathbf{S4DT_1S}$
is complete with respect to all finite DS-$T_1$-frames. If we
prove that for an arbitrary finite DS-$T_1$-frame $F=(W,R,R_D)$,
$$ \X\cpmor F$$ then we prove the theorem.

We use induction on the size of $F$. Consider three cases.

\textsf{Case I.} $W=R(w_0)$, $R_D=W\times W$ for some $w_0$. Since
$\mathbf{S4}$ is compete with respect to $\X$(cf.
\cite{Aiello}) and $F^-=(W,R)$ is $\mathbf{S4}$-frame, then
there exists a continuous function $f:\X \to Top(F^-)$. It is
easy to check that $f:\X \cpmor F$.

\textsf{Case II.} $W=R(w_0)$, $R_D=W\times W -(w_0,w_0)$. Since
$W$ is finite let us numerate all points in $W$ starting with
$w_0$: $W=\{w_0, w_1, w_2, \ldots w_n\}$. Any generated subframe
$F^{w_i}$ for $i>0$ satisfies to case I.

Take an arbitrary point $x_0$ and clopen sets $Y_0, Y_1, \ldots $
such that
$$
\{x_0\}\subset\ldots\subset Y_n\subset\ldots \subset Y_1\subset
Y_0=\X
$$
and
$$
Y_n \subseteq O(x_0, \frac{1}{n})
$$
for every $n>0$. We can do it because $\X$ is zero-dimensional
(cf. \cite{Rasiowa})

Since $Y_n \subseteq O(x_0, \frac{1}{n})$, it follows that
$$
\bigcap_n Y_n = \{x_0\}
$$
and further we obtain
$$
\X - \{x_0\} = \X - \bigcap_n Y_n = \bigsqcup_n \X_n,
$$
where $\X_n = Y_n-Y_{n+1}$. The sets $\X_n$ are open,
metric, dense-in-itself and zero-dimensional.

For any open neighborhood $U$ of $x_0$ there exists $n$ such that
$O(x_0,\frac 1n) \subset U$, it follows that $Y_n\subset U$, hence
for all $i\ge n$ $\X_i\subset U$.

So, by induction, for any $j>0$ there exists
$$
f_j: \X_j\cpmor F^{w_k}, \mbox{ where } (k-1) \equiv j
\;(\mathrm{mod} \; n)
$$
Now consider
$$
f(x)=\left\{
\begin{array}{ll}
w_0,&\mbox{if } x=x_0;\\
f_j(x),&\mbox{if } x\in\X_j.
\end{array}
\right.
$$

Let us prove that $f:\X\cpmor F$.

First, we note that $f$ is surjective.

Second, we check (\ref{RCf}). Assume that $y\in \X_j$ then:
$$
y\in \Cf f^{-1}(w) \Longrightarrow y\in \Cf f_j^{-1}(w)=
f_j^{-1}(R^{-1}(w))\subseteq f^{-1}(R^{-1}(w));
$$
and the other way around:
$$
y\in f^{-1}(R^{-1}(w))\Longrightarrow f_j^{-1}(R^{-1}(w))= y\in
\Cf f_j^{-1}(w)\subseteq \Cf f^{-1}(w).
$$
Now assume that $y = x_0$. For any $w\in W$, $w_0\in R^{-1}(w)$;
hence
$$
x_0\in f^{-1}(R^{-1}(w)).
$$
On the other hand, for some $i$, $w=w_i$ and for any open
neighborhood of $x_0$, there exists $m$ such that $(i-1)  \equiv m
\;(\mathrm{mod} \; n)$ and $U \supset \X_m \supset
f_m^{-1}(w_i)$. In other words, $x_0$ is a limit point for
$f^{-1}(w_i)$, hence $x_0\in\Cf f^{-1}(w_i)$.

Third, we check (\ref{RDf}).  Since $f^{-1}(w_0)$ is a
one-element set, (\ref{RDf}) holds.

\textsf{Case III.}  Everything else. Let us take all $R$-minimal
$R$-clusters of $F$ and from each one of them we choose an
arbitrary point. So we get the following set: $\{v_1, v_2, \ldots
, v_k\}$. Standard unravelling arguments show that
$$
F'=F^{v_1}\sqcup F^{v_2} \sqcup \ldots \sqcup F^{v_n}
\twoheadrightarrow F
$$
So we need to show that $\X \cpmor F'$.

Since $F$ is a $\mathbf{S4DT_1S}$-frame, each $F^{v_i}$ satisfies
case I or case II.

Since $\X$ is zero-dimensional, we can present $\X$ as
disjunctive union of clopen subsets:
$$
\X =  \X_1 \sqcup \ldots \sqcup \X_{k-1} \sqcup \X_k.
$$

By induction we have
$$
f_1: \X_1 \cpmor F^{v_1},
$$$$
\cdots
$$$$
f_{k-1}: \X_{k-1} \cpmor F^{v_{k-1}},
$$$$
f_{k}: \X_{k} \cpmor F^{v_{k}}
$$

It is easy to show that $f=f_1 \sqcup \ldots \sqcup f_k$ (if $x\in
\X_i$ then $f(x)=f_i(x)$) is a cd-$p$-morphism.
\end{proof}

The immediate and obvious corollary of this theorem is that
$\mathbf{S4DT_1S}$ is complete with respect to all dense-in-itself
$T_1$ spaces.

\section{Conclusions and open problems.}

The language with difference modality shows much more expressive
power then basic topological language, and even more then basic
language with universal modality. We can express
density-in-itself, $T_1$, and connectedness in it. Moreover the
axiom
$$
(AE_1)\qquad \D p \land \neg p \land \Box(p\to \Box
q\lor\Box\neg q) \to \Box(p \to q) \lor \Box(p \to \neg q)
$$
differs $\mathbb{R}$ from $\mathbb{R}^2$ (cf. \cite{Kudinov}).
It was proved that logic
$\mathbf{S4DT_1S}+(AE_1)+\mbox{``connectedness"}$ is complete with
respect to $\mathbb{R}^n$, $n\ge 2$ (the full proof is to be
published). We still do not know the D-logic of $\mathbb{R}$ and
whether $\mathbf{S4D}+(AT_1)$ is complete with respect to all
$T_1$ spaces.

\bigskip

\hrule

\bigskip

This paper was published at Advances in Modal Logic, Volume 6, 2006

\end{document}